\documentclass[11pt]{amsart}
\usepackage{amsthm}
\usepackage{amssymb}
\usepackage{amsfonts}
\usepackage{graphicx}
\usepackage{hyperref}
\usepackage{ytableau}
\usepackage[margin=1in]{geometry}
\hypersetup{
    colorlinks=true,
    urlcolor=blue,
    linkcolor=black,
    citecolor=black,
    filecolor=black
}

\usepackage[utf8]{inputenc}
\usepackage{cleveref}
\usepackage{thm-restate}
\usepackage[
backend=biber,
style=alphabetic,
sorting=nyt
]{biblatex}
\renewbibmacro{in:}{}
\DeclareFieldFormat{pages}{#1}
\DeclareFieldFormat[article,unpublished]{title}{#1}

\newtheorem{theorem}{Theorem}[section]
\newtheorem{prop}[theorem]{Proposition}
\newtheorem{cor}[theorem]{Corollary}
\newtheorem{conj}[theorem]{Conjecture}
\newtheorem{quest}[theorem]{Question}

\newtheorem{lem}[theorem]{Lemma}
 
\newtheorem{convention}[theorem]{Convention}

\theoremstyle{definition}
\newtheorem{defin}[theorem]{Definition}
\newtheorem{remark}[theorem]{Remark}
\newtheorem{exam}[theorem]{Example}

\newcommand{\Pln}{\mathcal{P}_\ell(n)}
\newcommand{\Pn}{\mathcal{P}(n)}
\newcommand{\pre}{\mathrm{pre}}
\newcommand{\prh}{\mathrm{prh}}
\newcommand{\prs}{\mathrm{prs}}
\newcommand{\ZZ}{\mathbb{Z}}

\title{Injectivity of symmetric polynomial maps on partitions}

\author{Rohith Thomas}
\address{Cherry Creek High School, Greenwood Village, CO.}
\email{\href{mailto:rohithkrishnathomas@gmail.com}{rohithkrishnathomas@gmail.com}}

\author{Katherine Tung}
\address{Department of Mathematics, Harvard University, Cambridge, MA.}
\email{\href{mailto:katherinetung@college.harvard.edu}{katherinetung@college.harvard.edu}}
\date{\today}

\begin{document}
\begin{abstract}

Introduced by Ballantine, Beck, and Merca, the elementary symmetric partition function $\pre_k$, defined on the set of partitions with at least $k$ parts, has been a topic of recent interest. We prove that $\pre_k$ is injective on the set of $m$-ary partitions for positive integers $m \ge k$, generalizing the binary $k = 2$ result of Ballantine, Beck, and Merca, and complementing a result of Hadelyn, Niergarth, Li and Li showing that, for each $k \ge 3$, $\pre_k$ is not injective on partitions of $n$ with length $2k$ for infinitely many $n$. We introduce the skew Schur partition function $\prs_{\lambda'/\mu'}$, prove injectivity results for particular choices of $\lambda',\mu'$, and describe an application to representation theory.

\end{abstract}
\maketitle

\section{Introduction}\label{sec:intro}

In this paper, we study certain functions constructed from integer partitions and symmetric polynomials, two classes of well-studied objects in combinatorics \cite{Andrews_1984,EC1,EC2,BBM, Sagan,Li,Bplus,DeEy,Garg,Ha+}.

First defined by Ballantine, Beck, and Merca \cite{BBM}, the elementary symmetric partition function $\pre_k$ is a map sending an integer partition $\lambda$ with at least $k$ parts to the summands of the elementary symmetric polynomial $e_k$ evaluated at $\lambda$. One may ask whether $\pre_k$ is injective. A back-of-the-envelope check reveals that $\pre_2(2,2) = \pre_2(4,1)$, so it is natural to filter by the size of the partitions before asking about injectivity. 

\begin{restatable}{quest}{mainQuest}\textup{\cite{BBM}}\label{quest: main}
For what positive integers $n$ and $k$ is the map $\pre_k$ injective on the set of partitions of $n$ of length greater than or equal to $k$?\end{restatable}
Devnani and Eyyunni~\cite{DeEy} gave an infinite family of examples demonstrating that $\pre_k$ is not injective on the set of partitions of length $k$. They also posed the following refinement of Question \ref{quest: main}.

\begin{conj}\textup{\cite{DeEy}}\label{conj: main}
For all positive integers $n$ and $k$, the map $\pre_k$ is injective on the set of partitions of $n$ with length strictly greater than $k$.
\end{conj}

Hadelyn et al.~\cite{Ha+} disproved this conjecture. 
We give a slightly different argument against the conjecture in Section~\ref{subsec: neg}.

\begin{restatable}{theorem}{mainTheorem}\label{theorem: main}
There exist infinitely many $n$ for which $\pre_3$ is not injective on partitions of $n$ with length greater than $3$.
\end{restatable}

However, there exist restricted classes of partitions on which $\pre_k$ is injective. For example, Ballantine, Beck, and Merca showed the injectivity of $\pre_2$ on the binary partitions of $n$. This result was generalized by Li, who showed injectivity of $\pre_2$ on all partitions of $n$ with at least 2 parts~\cite{Li}. 
As another example, Ballantine et al. showed that $\pre_k$ is injective on a family of partitions $\mathcal{S}_k$ consisting of partitions with at least $k$ $1$'s or $(k - 1)$ $1$'s and at least one prime part \cite{Bplus}.

Extending this line of research, we prove injectivity on the $m$-ary partitions, namely partitions whose parts are all powers of $m$, of $n$ for $m\geq k$:
\begin{restatable}{theorem}{maryTheorem}\label{theorem: mary}
For integers $m\geq k\geq 1$ and $n\geq 0$, the map $\pre_k$
is injective on the set of $m$-ary partitions of $n$.
\end{restatable}

In the literature, there are a variety of techniques used to establish injectivity. For instance, to show injectivity of $\pre_2$, Li~\cite{Li} utilized the fundamental structure of $e_2$ as well as the extremal principle. By contrast, Ballantine et al.~\cite{Bplus} and Ballantine, Beck, and Merca~\cite{BBM} adopted a more algorithmic approach when proving their results, attempting to construct $\lambda$ from $\pre_k(\lambda)$. We use a variety of approaches throughout the paper to prove our injectivity results.

\begin{theorem}\label{theorem: prh}
The map $\prh_k$ is injective on partitions of $n$ for all positive integers $n$ and $k$.
\end{theorem}
Hadelyn et al.~\cite{Ha+} recently proved the above result, but we offer an independently derived proof. 
Furthermore, we offer an explicit algorithm to determine $\lambda$ from $\prh_k(\lambda)$ whenever $\lambda$ contains a positive number of $1$'s.

In addition to considering injectivity, Devnani and Eyyunni~\cite{DeEy} considered how many partitions of $n$ are in the image of $\pre_2(\lambda)$, finding at least one for each positive integer $n$. They defined $\pre_2(n)$ to be this number. They then asked the following question:
\begin{restatable}{quest}{deQuest}\textup{\cite{DeEy}}\label{quest: P3}
Does there exist an $n\geq 1$ such that $\pre_2(n)=1$?
\end{restatable}
We fully address this question, as well as a generalization of it, in Section~\ref{subsec: good}. In his recent note, Garg~\cite{Garg} also addressed this question, but we offer our own proof, derived independently of him, to showcase a different technique.
We utilize a recursive style argument rather than direct casework, as used in Garg~\cite{Garg}.

We also prove injectivity for the $\prs_{\lambda'}$ function defined in Hadelyn et al. when $\lambda' = (2,1)$, thereby making progress on their Question 5.3 \cite{Ha+}.
\begin{restatable}{theorem}{prsTheorem}\label{theorem: prs_(2,1)}
The map $\prs_{(2,1)}$ is injective on partitions of $n$ for all positive integers $n$.
\end{restatable}

\subsection{Skew Schur polynomials}
We can generalize $\prs_{\lambda'}$ from traditional Schur polynomials over an integer partition $\lambda'$ to skew Schur polynomials over a skew shape $\lambda'/\mu'$. Thus, we define the analogous map $\prs_{\lambda'/\mu'}$. The value of this map derives from considering the geometry of a skew shape through the combinatorial definition of $s_{\lambda'/\mu'}$, allowing us to prove the following theorem:
\begin{restatable}{theorem}{skewTheorem}\label{theorem: skewTheorem}
For any skew shape $\lambda'/\mu'$ in which there is at most one square in each column and at most one square in each row, $\prs_{\lambda'/\mu'}$ is injective on $\Pn$.  
\end{restatable}

\subsection{Applications}
The elementary symmetric partition function has been previously studied in the literature, though not in this exact language. For example, in 1958, Selfridge and Strauss proved that if $X = \{x_1,\dots,x_n\}$ is a set of complex numbers, and $X' = \{x_i+x_j\}_{1 \le i < j \le n}$ is the set of pairwise sums of $X$, we can uniquely retrieve $X$ from $X'$ if $n$ is not a power of 2 \cite{SelfStraus}. This result implies that $\pre_2$ is injective on the set of partitions whose length is not a power of $2$.

Our results may also be viewed through a representation-theoretic lens. See  
Section~\ref{subsec:app} for a discussion of plethysm-related consequences of our results.

\subsection{Organization of the paper}\label{sec:outline}
The remainder of this paper is organized as follows:
\begin{itemize}
\item Section~\ref{sec:background} establishes preliminary definitions, notational conventions, and summarizes prior results on Question~\ref{quest: main}.
\item Section~\ref{sec: pre} concerns $\pre_k$, giving a proof of Theorem \ref{theorem: mary} and detailing some instances where injectivity holds and fails to hold. 
\item Section~\ref{sec: prh} concerns $\prh_k$, giving an alternate proof of injectivity and a remark on the existence of an algorithm.
\item Section~\ref{sec: prs} details progress on Schur polynomials—notably a proof of Theorem~\ref{theorem: prs_(2,1)}—and expands upon applications. 
\item Section~\ref{sec: openquestions} contains open problems and future directions. 
\end{itemize}

\section{Background}\label{sec:background}

\subsection{Key definitions}\label{subsec: defin}

We first introduce the definitions and notation regarding $\pre_k$ that are relevant to this paper.

A partition $\lambda=(\lambda_1,\lambda_2,\lambda_3,\dots,\lambda_\ell)$ of $n$ is a nonincreasing sequence of positive integers that sum to $n$. The value $n$ is also called the \emph{size} of $\lambda$, and $\ell$ is called the \emph{length} of $\lambda$. We sometimes also write the size of $\lambda$ as $|\lambda|$ and the length of $\lambda$ as $\ell(\lambda)$. Let $\mathcal{P}$ be the set of all partitions, let $\Pn$ be the set of all partitions of $n$, and let $\Pln$ be the set of all partitions of $n$ of length $\ell$.

\begin{convention}
Define $\mathcal{P}_{\ge k}(n) := \mathcal{P}_k(n) \cup \mathcal{P}_{k+1}(n) \cup \cdots \cup \mathcal{P}_n(n)$. Later in this section, we define some functions indexed by $k$ with domain $\mathcal{P}_{\ge k}(n)$ for various values of $n$. For brevity, we sometimes abuse notation and write the domains as $\Pn$ instead of $\mathcal{P}_{\ge k}(n)$.
\end{convention}

\begin{convention}\label{remark: conjpart}
In the literature, the notation $\lambda'$ is sometimes reserved for the conjugate partition of $\lambda$. We do not adopt this convention and instead use $\lambda'$ to denote some other partition in statements where $\lambda$ is already used. 
We use $\lambda^T$ for the conjugate of $\lambda$.  

\end{convention}

\begin{defin}
    Let the \textit{multiplicity} $m_\lambda(a)$ of an integer $a$ in $\lambda$ be the number of times $a$ appears in $\lambda$. 
\end{defin}

We use a variant of exponential notation for partitions, in which

\[\lambda=({d_1,\dots,d_1},{d_2,\dots d_2},\dots,{d_q,\dots,d_q})\] with no two $d_i$ equal can be abbreviated as $d_1^{m_\lambda(d_1)}d_2^{m_\lambda(d_2)}\cdots d_q^{m_\lambda(d_q)}$. For instance, $(3,3,2,2,2,1)$ is abbreviated as $3^22^31^1$, with exponents of $1$ sometimes omitted.

\begin{defin}\label{defin: elempoly}
For positive integers $j,\ell$ such that $j\leq \ell$, the \emph{$j$-th elementary symmetric polynomial} is the polynomial $e_j\in \mathbb{Z}[x_1,x_2,\dots,x_\ell]$ defined by:
\[e_j(x_1,x_2,\dots,x_\ell)=\sum_{1\leq i_1<i_2<\dots<i_j\leq \ell} x_{i_1}x_{i_2}\dots x_{i_j}.\]
\end{defin}

When we say $e_k$ is evaluated at a partition $\lambda$, we mean $e_k(\lambda_1,\lambda_2,\dots,\lambda_\ell)$, where $\ell$ is the length of $\lambda$. Using these definitions, we can now define the map $\pre_k$:
\begin{defin}\cite{BBM}\label{defin: prek}
For a partition $\lambda=(\lambda_1,\lambda_2,\dots,\lambda_\ell)$ of $n$ and a positive integer $k\leq \ell$, we define $\pre_k$ as the map such that $\pre_k(\lambda)$ is the partition composed of the summands of the $k$-th elementary symmetric polynomial evaluated at $\lambda$.
\end{defin}
\begin{exam}
    When $\lambda = (4,2,1,1)$ and $k = 2$, 
    \begin{align*}
        \pre_2(\lambda) &= (\lambda_1\lambda_2, \lambda_1 \lambda_3, \lambda_1\lambda_4, \lambda_2\lambda_3, \lambda_2\lambda_4, \lambda_3\lambda_4) \\ 
        &= (8,4,4,2,2,1).
    \end{align*}
\end{exam}
\begin{exam}\label{exam: prek}

In general, the relative sizes of the parts of $\lambda$ affects the ordering of the parts of $\pre_k(\lambda)$. 
For example, 
$\lambda = (4,3,2,1)$ satisfies $\lambda_2\lambda_3 \geq \lambda_1\lambda_4$, whereas $\lambda' = (7,1,1,1)$ satisfies $\lambda'_2\lambda'_3 \le \lambda'_1\lambda'_4$.
\end{exam}

We now discuss the complete homogeneous symmetric polynomial and the map $\prh_k$, the latter of which was first defined in \cite{Ha+}.

\begin{defin}\label{defin: homopoly}
For positive integers $j,n$, 
the \emph{$j$-th complete homogeneous symmetric polynomial} is the polynomial $h_j\in \mathbb{Z}[x_1,x_2,\dots,x_n]$ defined by:
$$h_j(x_1,x_2,\dots,x_n)=\sum_{1\leq i_1\leq i_2\leq \dots\leq i_j\leq n} x_{i_1}x_{i_2}\dots x_{i_j}.$$
\end{defin}

\begin{defin}\cite{Ha+}\label{defin: prh}
For a partition $\lambda=(\lambda_1,\lambda_2,\dots,\lambda_\ell)$ of $n$ and a positive integer $k$, we define $\prh_k$ as the map such that $\prh_k(\lambda)$ is the partition composed of the summands of the $k$-th complete homogeneous symmetric polynomial evaluated at $\lambda$.
\end{defin}
\begin{exam}
Let $\lambda=(2,2)$ and $\mu=(4,1)$. By our definition of $h_k$, we have $h_2(x_1,x_2) = x_1^2 + x_1x_2 + x_2^2$. Note that $\prh_2(\lambda) = (4,4,4)$ and $\prh_2(\mu) = (16,4,1)$. 
\end{exam}

Finally, we define the skew Schur polynomial $s_{\lambda/\mu}$ and the map $\prs_{\lambda' /\mu'}$. 

\begin{defin}\label{defin: young}

A \emph{Young diagram} for a partition $\lambda = (\lambda_1,\dots,\lambda_\ell)$ is a collection of boxes arranged in left-justified rows with $\ell$ total rows, and with row $i$ having $\lambda_i$ boxes for each $i$.
\end{defin}

\begin{defin}\label{defin: skew}
The \textit{skew shape} $\lambda/\mu$ is a pair of partitions $\lambda,\mu$ with $\ell(\lambda) \ge \ell(\mu)$ and $\lambda_i \ge \mu_i$ for all $i \in \{1,\dots,\ell(\mu)\}$. The \textit{skew diagram} of $\lambda/\mu$ is
the shape defined by the set of squares that belong to the Young diagram of $\lambda$ but not of $\mu$. 
\end{defin}

\begin{defin}\label{defin: SSYT}
A \textit{skew semi-standard Young tableau} of shape $\lambda/\mu$ is a filling of the squares of $\lambda/\mu$ with positive integers satisfying the following two conditions:
\begin{enumerate}
\item the entries are weakly increasing along every row and 
\item the entries are strictly increasing down every column.
\end{enumerate}
When $\mu = \emptyset$, we use the terminology ``semi-standard Young tableau (SSYT) of shape $\lambda$."
\end{defin}
\begin{exam}
Here is a (semi-)standard skew Young tableau of shape $\lambda /\mu$ for $\lambda = (5,4,2,2)$, $\mu = (2,1)$.
    
    \[
\begin{ytableau}
 \none & \none &  2&5 & 9 \\ 
 \none & 3 & 4 & 6 \\
 1 & 7 \\
 8 & 10\\
\end{ytableau}
\]
\end{exam}

Let a given skew semi-standard Young tableau $T$ of shape $\lambda /\mu$ have \textit{type} $\alpha=(\alpha_1,\alpha_2,\dots)$, where $\alpha_i$ is the number of entries in the tableau equal to $i$. Write $x^T=x_1^{\alpha_1}x_2^{\alpha_2}\dots$. We now define the skew Schur polynomial.

\begin{defin}
The \textit{skew Schur polynomial} $s_{\lambda/\mu}(x)$ for skew shape $\lambda/\mu$ and $x=(x_1,x_2,\dots,x_\ell)$ is defined as
\[s_{\lambda/\mu}(x)=\sum_{T} x^T\]
where the sum runs over all SSYTs $T$ of shape $\lambda/\mu$ with entries in $\{1,2,\dots,\ell\}$. Note that $\ell$ does not have to equal $|\lambda| - |\mu|$, but $\ell$ must be at least the maximum column height of $\lambda /\mu$. We define the usual {Schur polynomial} $s_\lambda$ as $s_{\lambda / \emptyset}$.
\end{defin}

\begin{exam}\label{exam: prs21}
Let $\lambda=(2,1)$ and consider the set of $n$ variables $x=(x_1,x_2,\dots,x_n)$. Now, consider three integers $1\leq i,j,k\leq n$ such that $i<j<k$. The SSYTs of shape $\lambda$ can be generated as follows:
\[
\begin{ytableau}
i&j \\ 
k\\
\end{ytableau}
\quad
\begin{ytableau}
i&k \\ 
j\\
\end{ytableau}
\quad
\begin{ytableau}
i&i \\ 
j\\
\end{ytableau}
\quad
\begin{ytableau}
i&j \\ 
j\\
\end{ytableau}
\] 

Thus, it follows:
\[s_{(2,1)}(x_1,\dots,x_n)=\sum_{1\leq i<j\leq n} (x_i^2x_j+x_ix_j^2)+2\sum_{1\leq i<j<k\leq n} x_ix_jx_k.\]
\end{exam}
With this framework, we define the maps $\prs_{\lambda'}$ and $\prs_{\lambda'/\mu'}$.

\begin{defin}\cite{Ha+}\label{defin: prs}
For a partition $\lambda=(\lambda_1,\lambda_2,\dots,\lambda_\ell)$ of $n$ and an arbitrary partition $\lambda'$ for which $s_{\lambda'}(\lambda)$ is defined and nonzero, we define $\prs_{\lambda'}(\lambda)$ to be the partition composed of the summands of $s_{\lambda'}(\lambda)$. Note that if $s_{\lambda'}$ has a term of coefficient $m>1$, then that term appears $m$ times in $\prs_{\lambda'}$.

\end{defin}

\begin{defin}\label{defin: prsskew}

For a partition $\lambda=(\lambda_1,\lambda_2,\dots,\lambda_\ell)$ of $n$ and an arbitrary skew partition $\lambda'/\mu'$ for which $s_{\lambda'/\mu'}(\lambda)$ is defined and nonzero, we define $\prs_{\lambda'/\mu'}(\lambda)$ to be the partition composed of the summands of $s_{\lambda'/\mu'}(\lambda)$. Note that if $s_{\lambda'/\mu'}$ has a term of coefficient $m>1$, then that term appears $m$ times in $\prs_{\lambda'/\mu'}$.

\end{defin}

\begin{remark}
Since $e_k = s_{(1^k)}$ and $h_k = s_{(k)}$, $\prs_{(1^k)} = \pre_k$ and $\prs_{(k)} = \prh_k$.

\end{remark}

\begin{lem}[Jacobi-Trudi identities]\label{lem: jacobitrudi}

Consider a partition $\lambda=(\lambda_1,\lambda_2,\dots,\lambda_\ell)$. Let the conjugate partition of $\lambda$ be $\lambda^T=(\lambda_1^T,\lambda_2^T,\dots,\lambda_{\ell'}^T)$, and let $\ell' = \ell(\lambda^T)$. 
We then have 
\[s_{\lambda}=\det(h_{\lambda_i+j-i})_{i,j=1}^{\ell},\]
\[s_{\lambda}=\det(e_{\lambda_i^T+j-i})_{i,j=1}^{\ell'},\]
where if $t = 0$, we set $h_t=e_t=1$, and if $t<0$, we set $h_t=e_t=0$.
Additionally, if we consider $s_\lambda$ over $m$ variables $x:=(x_1,x_2,\dots,x_m)$, then we have 
\[s_{\lambda}=\frac{\det(x_j^{\lambda_i+m-i})_{i,j=1}^m}{\det(x_j^{m-i})_{i,j=1}^m}.\]
Note that if $i>\ell$, then $\lambda_i$ is set to be $0$.
\end{lem}

\begin{convention}
    Let $f$ be $\pre_k$, $\prh_k$, or $\prs_{\lambda'/\mu'}$. Given some input partition $\lambda$, we usually denote $f(\lambda)$ by $\nu$. 
\end{convention}
\begin{convention}
    We set $\binom{n}{k} = 0$ if $n<k$.
\end{convention}

  \begin{prop}\textup{\cite{Bplus}}
      Let $\lambda$ be a partition with $m_\lambda(1) > 0$, and let $\nu = \pre_k(\lambda)$. Then $m_{\nu}(1)=\binom{m_\lambda(1)}{k}$.
  \end{prop}

\subsection{Other definitions}
\begin{defin}\cite{Bplus}\label{defin: classofpartition}
Define $\mathcal{S}_k$ as the set of partitions  such that, for each $\lambda\in \mathcal{S}_k$, either the number of $1$s in $\lambda$ is at least $k$, or the number of $1$s in $\lambda$ is $k-1$ and there exists some prime $p$ for which the number of $p$s in $\lambda$ is positive.
\end{defin}
\begin{defin}
Define the \textit{$k$-ary partitions} of $n$, denoted $\mathcal{B}_k(n)$, as the set of all partitions of $n$ whose parts are all powers of $k$.
\end{defin}
The latter definition is a slight modification of a notational convention in \cite{BBM}, where $\mathcal{B}_2(n)$ is simply denoted as $\mathcal{B}(n)$.

\begin{defin}\cite{DeEy}\label{defin: pren}
Let $Pre_k(n)$ be the set of partitions of $n$ that lie in the image of $\pre_k(\lambda)$ for $\lambda\in\mathcal{P}$. Define $\pre_k(n)$ as the cardinality of $Pre_k(n)$.
\end{defin}

\subsection{Conventions}\label{subsec: notations}

As noted by Ballantine, Beck, and Merca~\cite{BBM}, if $\lambda$ and $\mu$ have different lengths, it is impossible for $\pre_k(\lambda)$ to equal $\pre_k(\mu)$. This property also applies to $\prh_k$ and $\prs_{\lambda'}$ in the cases where we need it. Thus, in our proofs of injectivity, we often start by assuming that if $f(\lambda) = f(\mu)$ for $f$ some function on partitions, then $\ell(\lambda) = \ell(\mu)$.

\section{Elementary symmetric polynomials}\label{sec: pre}
\subsection{Non-injective families of partitions}\label{subsec: neg}
Note that Theorem \ref{theorem: main} was proven by Hadelyn et al.~\cite{Ha+} where they proved the more general case that there are infinitely many pairs of partitions $\lambda,\mu$ of length $2j$ for which $\pre_j(\lambda)=\pre_j(\mu)$. However, we give our progress as it was derived independently of them. We first disprove Conjecture~\ref{conj: main} by generating an infinite family of counterexamples.  To motivate this, we provide an explicit counterexample:
\begin{exam}\label{exam: counter}
Consider $\lambda = (40,16,16,16,5,5)$ and $\mu = (32,32,10,10,10,4)$. Note that both $\lambda$ and $\mu$ are partitions of $98$. Furthermore, it follows:
\[\pre_3(\lambda)=10240^34096^13200^61280^61000^1400^3=\pre_3(\mu).\]
Thus, $\pre_3$ is not injective on $\mathcal{P}(98)$.
\end{exam}

\mainTheorem*

\begin{proof}
Consider $\lambda:=(a,b,b,b,c,c)$ and $\mu:=(d,d,e,e,e,f)$ for $a,b,c,d,e,f\in\ZZ^{+}$. It follows that
\[\pre_3(\lambda)=(ab^2)^3(abc)^6(ac^2)^1(b^3)^1 (b^2c)^6 (bc^2)^3,\]
\[\pre_3(\mu)=(d^2e)^3(def)^6(d^2f)^1(e^3)^1(de^2)^6(e^2f)^3.\]

In order to directly compare these two partitions in an efficient manner, we can set terms of the same multiplicity equal. Mimicking the structure of the given counterexample, we obtain the following system of equations:
\[ab^2=d^2e,\]
\[b^3=d^2f,\]
\[bc^2=e^2f,\]
\[ac^2=e^3,\]
\[abc=de^2,\]
\[b^2c=def.\]
Using these six equations, we can restrict the problem to three degrees of freedom:
\[a=\frac{e^3}{c^2}, b=\frac{e^2f}{c^2}, d=\frac{e^3f}{c^3}.\]
Now, as $\lambda$ and $\mu$ are partitions of the same number, it follows that $a+2c+3b=2d+3e+f$. Substituting our achieved values for $a,b,d$, we can simplify this equation as follows:
\[f=\frac{3ec^3-2c^4-ce^3}{3ce^2-c^3-2e^3}.\]

Let $c=k$ and $e=2k$ for $k\in \ZZ^+$. Then, $f$ simplifies to $\frac45k$. Thus, let $k=5q$ for positive integers $q$, so $f=4q$. Thus, $a=40q, b=16q, d=32q$, all of which are integers. Thus, by varying $q$ over the positive integers, we generate infinitely many partitions $\lambda,\mu\in\Pn$ for which $\pre_k(\lambda)=\pre_k(\mu)$.
\end{proof}

\begin{remark}\label{remark: othercounter}
Note that there are many more counterexamples of this type. For instance, if $\lambda=(189, 45, 45, 45, 7, 7)$ and $\mu=(135, 135, 21, 21, 21, 5)$, we may calculate that $\lambda,\mu\in \mathcal{P}(338)$ and $\pre_3(\lambda)=\pre_3(\mu)$. 
\end{remark}

\begin{remark}\label{remark: P3}
The existence of such counterexamples motivates a new question: how often is $\pre_k$ injective? Devnani and Eyyunni~\cite{DeEy} asked this as an open question for the case of $\mathcal{P}_k(n)$. It turns out that in this scenario, $\pre_k$ is ''almost" never injective. Particularly, for sufficiently large $n$ relative to $k$, $\pre_k$ fails to be injective on $\mathcal{P}_k(n)$. We prove this now, starting with the case of $k=3$. 
\end{remark}

\begin{lem}\label{lem: counter3}
The map $\pre_3$ is not injective on partitions of $n$ with $3$ parts when $n > 18$.
\end{lem}
\begin{proof}
Consider any prime $p\geq 13$. Recall that every prime $p\geq 3$ is of the form $6q+1$ or $6q-1$ for some integer $q$.

Suppose $p=6q+1$. Thus, consider $\lambda=(q,2q-2,3q+3)$ and $\mu=(q-1,2q+2,3q)$, both of which are partitions of $p$. Note the following equality:
\[\pre_3(\lambda)=q(2q-2)(3q+3)=(q-1)(2q+2)(3q)=\pre_3(\mu).\]
Thus, $\pre_3$ is not injective on $\mathcal{P}_3(p)$.

Now, suppose $p=6q-1$. Using an argument similar to above, letting $\lambda=(q,2q+2,3q-3)$ and $\mu=(q+1,2q-2,3q)$, it follows that $\pre_3(\lambda)=\pre_3(\mu)$. Thus, $\pre_3$ is not injective on $\mathcal{P}_3(p)$ for $k\geq 2$. Therefore, $\pre_3$ is not injective on $\mathcal{P}_3(p)$ for any prime $p\geq 13$.

Furthermore, by manually checking, we can see that $\pre_3$ is not injective on $\mathcal{P}_3(n)$ when $n=16, 25, 27,$ and $49$.

Now consider two partitions $\lambda=(a_1,b_1,c_1), \mu=(a_2,b_2,c_2)\in \mathcal{P}(n)$ such that $\pre_3(\lambda)=\pre_3(\mu)$ and $n=16,25,27,49,$ or a prime $p\geq 13$. Note, for any positive integer $k$, it follows that $k\lambda=(ka_1,kb_1,kc_1)$ and $k\mu=(ka_2,kb_2,kc_2)$ are partitions of $kn$ satisfying $\pre_3(k\lambda)=\pre_3(k\mu)$. Therefore, $\pre_3$ is not injective on $\mathcal{P}_3(kn)$. Thus, it follows that $\pre_3$ is not injective on partitions of any number of the form $2^{a_1}\cdot 3^{a_2}\cdot 5^{a_3}\cdot 7^{a_4}\cdot 11^{a_5}\cdots$ where $a_1\geq 4, a_2\geq 3, a_3\geq 2, a_4\geq 2,$ and $a_i\geq 0$ for $i\geq 5$. For the numbers with $a_1< 4, a_2< 3, a_3< 2, a_4<2$ and $a_i=0$ for $i\geq 5$, we can manually check their injectivity on $3$ parts to see that $\pre_3$ is not injective when $n=14,16$ and when $n>18$. Thus, $\pre_3$ is injective if and only if $3\leq n\leq 12$ and $n=15,18$.
\end{proof}

With this result, we can now prove the following theorem:
\begin{theorem}\label{theorem: counter}
When $k \ge 3$, the map $\pre_k$ is not injective on partitions of $n$ with $k$ parts when $n > k + 15$.
\end{theorem}
\begin{proof}
For integers $k\geq 3$ and $n\geq k$, consider partitions $\lambda,\mu\in \mathcal{P}_3(n-k+3)$ such that $\pre_3(\lambda)=\pre_3(\mu)$, which exist by the previous result. Now consider $\lambda',\mu'\in \mathcal{P}(n)$ which are generated by adding $k-3$ $1$'s to $\lambda,\mu$, respectively. It follows that $\pre_k(\lambda')=\pre_k(\mu')$. Thus, by the previous result, it follows $\pre_k$ is not injective on partitions of $n>k+15$ of length $k$, as desired.
\end{proof}

\begin{remark}\label{remark: compcheck}
Through a computational check, one can make the observation that there are exactly $8$ values of $n$ for which $\pre_k$ is injective on $\mathcal{P}_k(n)$ for $k\geq 5$. This empirically suggests that $\pre_k$ is essentially never injective on $\mathcal{P}_k(n)$. 
\end{remark}

Furthermore, Devnani and Eyyunni~\cite{DeEy} offered the question of whether a lattice type technique could be feasible for proving injectivity of $\pre_2$. They offered such a technique for partitions of length $4,5,$ and $6$, and asked the open question of whether this is possible for other lengths. We offer a intuitive reason for why it is likely not possible:
\begin{remark}\label{remark: latticefail}
There are only two possibilities that could arise from employing this lattice-type method: one-by-one casework or a nice set of cases which eliminates all others. 

The former approach is computationally unfeasible. Particularly, there are $\binom{\ell}{2}$ elements in $\pre_2$, requiring exponential time to check each case. Note that this method was what \cite{DeEy} used. Thus, this is infeasible if we were to do so for general $\ell$. 

In the latter case, this would mean injectivity would follow under a significantly smaller set than all of $\pre_2$. One way to think about this is by looking at each 'level' of the lattice structure, and seeing if that is sufficient to show injectivity. However, as shown by empirical results, this is not sufficient to show injectivity on all of $\pre_2$.
\end{remark}

\subsection{Injective families of partitions}\label{subsec: good}
While the conjecture guiding this research problem, Conjecture~\ref{conj: main}, turned out to be false, there are still many positive results associated with injectivity on $\pre_k$. We begin by extending Lemma 2.1 of Li~\cite{Li}.
\begin{lem}\label{lem: liextension}
For integer partitions $\lambda=(\lambda_1,\lambda_2,\dots,\lambda_\ell)$ and $\mu=(\mu_1,\mu_2,\dots,\mu_{\ell})$ of $n$, if $\lambda_i=\mu_i$ for $1\leq i\leq k-1$ and $\pre_k(\lambda)=\pre_k(\mu)$, then $\lambda=\mu$.
\end{lem}
\begin{proof}
We use strong induction on $\ell$. For our base case, suppose $\ell=k$. Note a maximum element in $\pre_k(\lambda)$ is $\lambda_1\lambda_2\dots\lambda_{k-1}\lambda_k$.
Also, a maximum element in $\pre_k(\mu)$ is
\[\mu_1\mu_2\dots\mu_{k-1}\mu_k=\lambda_1\lambda_2\dots\lambda_{k-1}\mu_k,\]
which follows from the inductive hypothesis $\lambda_i=\mu_i$ for all $1\leq i\leq k-1$. As $\pre_k(\mu)=\pre_k(\lambda)$, the maximum elements of each set will be equal. Thus, $\lambda_k=\mu_k$.

Now, assume $\lambda_i=\mu_i$ for all $1\leq i\leq k'-1$ and that $k'\geq k+1$. Consider the following multiset:
\[M_{k,k'}(\lambda):=\{\{\lambda_{i_1}\lambda_{i_2}\dots\lambda_{i_k}: 1\leq i_1<i_2<\dots<i_k\leq k'-1\}\}.\] 
Thus, by the inductive hypothesis, $M_{k,k'}(\lambda)=M_{k,k'}(\mu)$. Therefore, it follows:
\[\pre_k(\lambda)\setminus M_{k,k'}(\lambda)=\pre_k(\mu)\setminus M_{k,k'}(\mu).\]
This means that the maximum elements of both of the above sets are equal.

Now, consider the element $a=\lambda_1\lambda_2\dots \lambda_{k-1}\lambda_{k'}\in \pre_k(\lambda)$, which we claim to be the maximum element of $\pre_k(\lambda)\setminus M_{k,k'}(\lambda)$. Consider an arbitrary element $s=\lambda_{i_1}\lambda_{i_2}\dots \lambda_{i_k}\in\pre_k(\lambda)$. Suppose for the sake of contradiction that $s$ is the maximum element. By definition, $\lambda_j\geq\lambda_{i_j}$. Note $\lambda_{k'}\geq \lambda_{i_k}$ if $i_k\geq k'$. Therefore, $i_k<k'$ is a necessary condition for $s>a$. However, if $i_k<k'$, then $s\in M_{k,k'}(\lambda)$, which is a contradiction. Thus, $a$ is a maximum element. Likewise, we can say that $\mu_1\mu_2\dots \mu_{k-1}\mu_{k'}$ is a maximum element of $\pre_k(\mu)\setminus M_{k,k'}(\mu)$.

Thus, it follows:
\[\lambda_1\lambda_2\dots \lambda_{k-1}\lambda_{k'}=\mu_1\mu_2\dots\mu_{k-1}\mu_{k'}=\lambda_1\lambda_2\dots\lambda_{k-1}\mu_{k'}.\]
Hence, $\lambda_{k'}=\mu_{k'}$. Therefore, by strong induction, we conclude $\lambda=\mu$, as desired.
\end{proof}

We now offer a proof of Theorem~\ref{theorem: length}, which was obtained independently of Devnani and Eyyunni~\cite{DeEy}, in order to offer an analog to a theorem we prove about Schur polynomials in Section~\ref{sec: prs}.

\begin{restatable}{theorem}{lengthTheorem}\cite{DeEy}\label{theorem: length}
For some $k$ such that $1\leq k \leq \ell-1$, the map $\pre_k$ is injective on $\mathcal{P}_\ell(n)$ if and only if the map $\pre_{\ell-k}$ is injective on $\mathcal{P}_\ell(n)$.
\end{restatable}
\begin{proof}
Suppose there exist two partitions $\lambda,\mu \in \mathcal{P}_\ell(n)$ such that $\pre_k(\lambda)=\pre_k(\mu)$. Note the product of all elements in $\pre_k(\lambda)$ is $(\lambda_1\lambda_2\dots\lambda_\ell)^{\binom{\ell-1}{k-1}}$. Likewise, the product of all elements in $\pre_k(\mu)$ is $(\mu_1\mu_2\dots\mu_\ell)^{\binom{\ell-1}{k-1}}$.
Clearly, both of these are equal, so $\lambda_1\lambda_2\dots\lambda_\ell=\mu_1\mu_2\dots\mu_
\ell$. 

Note $\pre_{\ell-k}(\lambda)$ can be generated by taking the reciprocal of every element in $\pre_k(\lambda)$ and then multiplying every element by $\lambda_1\lambda_2\dots\lambda_\ell$. Doing a similar operation for generating $\pre_{\ell-k}(\mu)$, by the assumption $\pre_k(\lambda)=\pre_k(\mu)$, we see that $\pre_{\ell-k}(\lambda)=\pre_{\ell-k}(\mu)$. Thus, $\pre_k$ being injective is equivalent to $\pre_{\ell-k}$ being injective, as desired.
\end{proof}
\begin{cor}\label{cor: lengthcase}
For positive integers $k$, the function $\pre_k$ is injective on partitions of length $k+1$ and length $k+2$.
\end{cor}

We now highlight the algorithmic approaches to injectivity. In their paper, Ballantine et al.~\cite{Bplus} gave, in Remark 2.7, an algorithm to determine $\lambda$ from $\nu=\pre_2(\lambda)$ for the class $\mathcal{S}_k$. We can also extend this to $k=3$ through the cubic formula:
\begin{exam}\label{exam: remark2.7to3}
Assume $m_\lambda(1)\geq k=3$. Recall the following inequality:
\[m_\nu(1)=\binom{m_\lambda(1)}{3}.\]
Letting $c=m_\nu(1)$, it follows:
\[m_\lambda(1)^3-3m_\lambda(1)^2+2m_\lambda(1)-6c=0.\]
Using the cubic formula, we obtain the following expression (which yields an integer if and only if $c = \binom{a}{3}$ for some integer $a$):
\[m_\lambda(1)=1+\left(3c+\sqrt{9c^{2}-\frac{1}{27}}\right)^{\frac{1}{3}}+\left(3c-\sqrt{9c^{2}-\frac{1}{27}}\right)^{\frac{1}{3}}.\]
Now, using an analogous method of Lemma 2.4 in \cite{Bplus}, we can develop a recursive relation for $m_\lambda(y)$, allowing us to determine $\lambda$ uniquely from $\nu$.

For the case of $m_\lambda(1)=k-1=2$ and $m_\lambda(p)\geq 1$ for some prime $p$, we can do an analogous procedure to above to determine $\lambda$ from $\nu$.
\end{exam}

Utilizing an algorithmic approach, we can also show injectivity on the class of the $m$-ary partitions, $\mathcal{B}_m(n)$, for $m\geq k$. Therefore, we will now prove Theorem~\ref{theorem: mary}.
\begin{convention}\label{convention: Sk}
Note, by the injectivity of $\pre_k$ on $\mathcal{S}_k$ \cite{Bplus}, if $m_\lambda(1)\geq k$, then $\lambda$ is unique to $\pre_k(\lambda)$. Thus, henceforth, assume $0\leq m_\lambda(1)<k$.
\end{convention}
We first consider the following lemma:
\begin{lem}\label{lem: k-ary1}
Assume $0\leq m_\lambda(1)<k$ and $0\leq m_\mu(1)<k$. For partitions $\lambda,\mu$ in $\mathcal{B}_m(n)$, we have $m_\lambda(1)=m_\mu(1)$.
\end{lem}
\begin{proof}
Suppose $\lambda=(m^{a_1},m^{a_2},\dots,m^{a_\ell})$ and $\mu=(m^{b_1},m^{b_2},\dots, m^{b_\ell})$. Note the following equality:
\[n=m^{a_1}+m^{a_2}+\dots+m^{a_\ell}=m^{b_1}+m^{b_2}+\dots+m^{b_\ell}.\]
Taking the equality modulo $m$, we see that
\[n\pmod{m}\equiv m_\lambda(1)\pmod{m}\equiv m_\mu(1)\pmod{m}.\]

Thus, $m_\lambda(1)=m_\mu(1)+Cm$ for some integer constant $C$. However, as we assumed $0\leq m_\lambda(1)\leq k-1$ and as $m\geq k$, it follows that $C=0$. Thus, $m_\lambda(1)=m_\mu(1)$, as desired.
\end{proof}

Consider the function $\varphi: \mathcal{P}(n)\to \mathbb{Z}$ such that $\varphi(\alpha)=\prod_{\text{distinct }a\in\alpha} \binom{m_\lambda(m^a)}{m_\alpha(a)}$. It follows:
\begin{equation}\label{eq:1}
m_\nu(m^q)=\sum_{\substack{\alpha\in \mathcal{P}(q)\\ k\geq \ell(\alpha)\geq k-m_\lambda(1)}} \varphi(\alpha)\cdot \binom{m_\lambda(1)}{k-\ell(\alpha)}.
\end{equation}

We proceed with an inductive type argument on $m_\lambda(m^q)$. By Lemma~\ref{lem: k-ary1}, we know that $m_\lambda(1)$ is uniquely determined by $\nu$ and $n$. 

Now, suppose we know what $m_\lambda(1),m_\lambda(m),\dots,m_\lambda(m^{q-1})$ are for an integer $q\geq 1$. Let $m_\lambda(m^i)=a_i$ for $0\leq i\leq q-1$. Consider the following summation:
\[S_{q-1}=m_\lambda(1)+m_\lambda(m)+\dots+m_\lambda(m^{q-1}).\]

Suppose $S_{q-1}<k$. Consider $m_\nu(m^{\sum_{i=0}^{q} i\cdot a_i})$ where $a_q=k-(a_0+a_1+\dots+a_{q-1})$. Note, a potential nonzero element in its summation, as per Equation~\ref{eq:1}, would be as follows:
\[\binom{m_\lambda(m^q)}{a_q}\cdot \prod_{i=0}^{q-1}\binom{m_\lambda(m^i)}{a_i}=\binom{m_\lambda(m^q)}{a_q}.\]
We claim this is the only one:

\begin{lem}\label{lem: equality}
The following equality holds:
\[m_\nu(m^{\sum_{i=0}^{q} i\cdot a_i})=\binom{m_\lambda(m^q)}{a_q}\cdot \prod_{i=0}^{q-1}\binom{m_\lambda(m^i)}{a_i}=\binom{m_\lambda(m^q)}{a_q}.\]
\end{lem}
\begin{proof}
Suppose for the sake of contradiction that in the summation for $m_\nu(m^{\sum_{i=0}^{q} i\cdot a_i})$, there exists another nonzero summand besides the one listed above, as per Equation~\ref{eq:1}. Suppose this other summand is of the following form:
\[\prod_{i=q}^\infty\binom{m_\lambda(m^i)}{b_i}\cdot \prod_{i=0}^{q-1}\binom{m_\lambda(m^i)}{b_i}\]
for some infinite sequence $(b_i)_{i\geq 0}$. Note $k=b_0+b_1+\dots$ and $b_i\leq a_i$ for all $0\leq i\leq q-1$, with at least one inequality being strict. Since we are evaluating $m_\nu(m^{\sum_{i=0}^{q} i\cdot a_i})$, we can make the following equality:
\begin{align*}
\sum_{i=0}^{q} i\cdot a_i&=\sum_{i=0}^{\infty} i\cdot b_i,\\
\sum_{i=q}^\infty i\cdot b_i&=q\cdot a_q+\sum_{i=0}^{q-1} i\cdot (a_i-b_i)\\
&=q\cdot (k-a_0-a_1-a_2-\dots-a_{q-1})+\sum_{i=0}^{q-1} i\cdot (a_i-b_i).\\
\end{align*}
However, note that the left-hand side of the above equation satisfies the following inequality:
\[\sum_{i=q}^\infty i\cdot b_i\geq q\cdot\sum_{i=q}^\infty  b_i=q(k-b_0-b_1-b_2-\dots-b_{q-1}).\]
Thus, it follows:
\begin{align*}
&q\cdot (k-a_0-a_1-a_2-\dots-a_{q-1})+\sum_{i=0}^{q-1} i\cdot (a_i-b_i)\geq q(k-b_0-b_1-b_2-\dots-b_{q-1}),\\
&\sum_{i=0}^{q-1} i\cdot (a_i-b_i)\geq \sum_{i=0}^{q-1} q\cdot (a_i-b_i),\\
&\sum_{i=0}^{q-1} (q-i)\cdot (a_i-b_i)\leq 0.\\
\end{align*}
Note that $a_i-b_i\geq 0$ for all $i$, with at least one being positive. Also, all of $q-i$ are positive. Therefore, the left-hand sum must be strictly positive, yielding a contradiction, as desired.
\end{proof}

Thus, if $m_\nu(m^{\sum_{i=0}^{q} i\cdot a_i})>0$, then, due to the strictly increasing nature of the binomial function, we can determine $m_\lambda(m^q)$. On the other hand, if $m_\nu(m^{\sum_{i=0}^{q} i\cdot a_i})=0$, then $m_\lambda(m^q)<a_q\leq k\leq m$. However, note the following equality:
\begin{align*}
n&= \sum_{i=0}^{\infty}m_\lambda(m^i)\cdot m^i\\
&\equiv m_\lambda(1)+m_\lambda(m)\cdot m+\dots+ m_\lambda(m^q)\cdot m^q\pmod{m^{q+1}}.
\end{align*}
Therefore, it follows:
\[m_\lambda(m^q)\equiv \frac{n-m_\lambda(1)-m_\lambda(m)\cdot m-\dots- m_\lambda(m^{q-1})\cdot m^{q-1}}{m^q}\pmod{m}.\]
As $m_\lambda(m^q)<m$, however, this means the congruence forces equality, allowing us to determine $m_\lambda(m^q)$.\\

Using this result, we can extract $m_\lambda(1), m_\lambda(m), \cdots$ until the summation $S_{q}$ is at least $k$. Suppose this occurs at $m_\lambda(m^p)$:
\begin{align*}
&S_{p-1}=m_\lambda(1)+m_\lambda(m)+\dots+m_\lambda(m^{p-1})<k,\\
&S_{p}=m_\lambda(1)+m_\lambda(m)+\dots+m_\lambda(m^{p-1})+m_\lambda(m^p)\geq k.
\end{align*}
Thus, we know the values of $m_\lambda(1),m_\lambda(m),\cdots, m_\lambda(m^p)$. Now, consider $m_\nu(m^{1+\sum_{i=0}^{p} i\cdot a_i})$. Using a similar method as in the proof of Lemma~\ref{lem: equality}, we can see that $m_\nu(m^{1+\sum_{i=0}^{p} i\cdot a_i})$ can be expressed in terms of solely $m_\lambda(1),m_\lambda(m),\cdots, m_\lambda(m^p),$ and $m_\lambda(m^{p+1})$. Furthermore, the only summands in $m_\nu(m^{1+\sum_{i=0}^{p} i\cdot a_i})$ that contain $m_\lambda(m^{p+1})$ would be of the form $\binom{m_\lambda(m^{p+1})}{j}\cdot X$ for some positive integer $j\leq m_\lambda(m^{p+1})$ and some positive integer $X$. Thus, due to the increasing nature of the binomial function, it follows that $m_\nu(m^{1+\sum_{i=0}^{p} i\cdot a_i})$ would be an increasing function of $m_\lambda(m^{p+1})$. Thus, if $m_\nu(m^{1+\sum_{i=0}^{p} i\cdot a_i})$ is nonzero, we can determine $m_\lambda(m^{p+1})$. 

Now, suppose $m_\nu(m^{1+\sum_{i=0}^{p} i\cdot a_i})=0$. Note that a nonzero element of the summation, as per Equation~\ref{eq:1}, is as follows:
\[m_\nu(m^{1+\sum_{i=0}^{p} i\cdot a_i})\geq m_\lambda(m^{p+1})\cdot \binom{m_\lambda(m^p)}{a_p-1}\cdot \prod_{i=0}^{p-1}\binom{m_\lambda(m^i)}{a_i}\geq m_\lambda(m^{p+1}).\]
Thus, it follows $m_\lambda(m^{p+1})=0$.\\

We can repeat this process inductively. Namely, we can write $m_\nu(m^{g+\sum_{i=0}^{p} i\cdot a_i})$, which, by Equation~\ref{eq:1} can be expressed in terms of solely $m_\lambda(1),m_\lambda(m),\cdots, m_\lambda(m^{p+g-1}),$ and $m_\lambda(m^{p+g})$. Thus, if $m_\nu(m^{g+\sum_{i=0}^{p} i\cdot a_i})\neq 0$, we can determine $m_\lambda(m^{p+g})$. If it were $0$, then, we can note the following inequality:
\[m_\nu(m^{g+\sum_{i=0}^{p} i\cdot a_i})\geq m_\lambda(m^{p+g})\cdot \binom{m_\lambda(m^p)}{a_p-1}\cdot \prod_{i=0}^{p-1}\binom{m_\lambda(m^i)}{a_i}\geq m_\lambda(m^{p+g}).\]
Thus, we can determine $m_\lambda(m^{p+g})$ to be $0$. Inducting this over $g\in \ZZ^+$, it follows that we can determine all of $\lambda$ from $\nu$, thus proving injectivity on the $m$-ary partitions:

\maryTheorem*
\begin{exam}\label{exam: kary}
To illustrate how we can determine $\lambda$ from $\nu$ for an $m$-ary partition, consider $n=28$, $m=k=3$, and $\nu=(243^3, 81^7, 27^6, 9^3)$. As $m_\nu(1)=0$, we can determine $m_\lambda(1)$ to be the residue of $n$ modulo $3$, which is $1$. Now, we can say $3=m_\nu(9)=\binom{m_\lambda(3)}{2}$. Thus $m_\lambda(3)=3$. Then, we can say $7=m_\nu(27)=m_\lambda(9)\cdot m_\lambda(3)+\binom{m_\lambda(3)}{3}$. Thus, $m_\lambda(9)=2$. We can stop here as the components of our partition sum to $28$, making the unique $\lambda$ be $(9,9,3,3,3,1)$.
\end{exam}
\begin{remark}\label{remark: marybad}
Extending this to the $m$-ary partitions in general, say if $m<k$, has shown to be difficult, if not impossible. In particular, there is not a neat way to extract a similar result as in Lemma~\ref{lem: k-ary1} to here. In the proof of Theorem~\ref{theorem: mary}, the condition $m \geq k$ is used critically at multiple stages, hinting that it is a necessary requirement.

Furthermore, unlike in the above scenario, if $m<k$, the $m$-ary partitions seem to not be injective on $k$ parts on $\pre_k$. For instance, consider $k = 5$ and the binary partitions $\lambda = (8, 8, 8, 1, 1)$ and $\mu = (16, 4, 2, 2, 2)$ of $26$. Note $\pre_5(\lambda)=512=\pre_5(\mu)$, so injectivity fails.

\end{remark}

This highlights the limitations of a purely algorithmic approach in proving general injectivity. Similar to Remark~\ref{remark: latticefail} on why a purely lattice-based approach to prove injectivity on $\pre_2$ is likely bound to fail, arguments trying to individually extract information about $\pre_k$ element by element are likely bound to fail for large cases. That is why Li's~\cite{Li} argument, which relied on the fundamental structure of $e_2$, allowed us to prove general injectivity for $\pre_2$. That being said, algorithmic approaches can show an approximation of the extent to which injectivity holds, such as on specific families of partitions.

We end this section on $\pre_k$ by answering Question~\ref{quest: P3}, as well as a generalization of it. Note that Garg~\cite{Garg} also proved the following result, but we provide it anyway for completeness and also because it was independently derived. 
\begin{theorem}\label{theorem: p3}
The only positive integers $n$ for which $\pre_2(n)=1$ are $n=1,2,$ and $4$.
\end{theorem}
\begin{proof}
One can easily check that $n=1,2,4$ satisfy $\pre_2(n)=1$, as these cases are relatively small. Thus, we now show that no other $n$ satisfy $\pre_2(n)=1$.

Henceforth, denote $\nu$ as a partition of $n$ and also as the output of $\pre_2(\lambda)$ for a partition $\lambda$. Recall that for any $n\geq 1$, a valid integer partition $\lambda$ such that $\nu$ is a partition of $n$ is $\lambda=(n,1)$, so $\pre_2(n)\geq 1$ for all such $n$. Thus, it suffices to show there exists some partition $\lambda\neq (n,1)$ such that $\pre_2(\lambda)=\nu$ is a partition of $n$. 

Take $\lambda=(x,1,1)$ for any positive integer $x$. Thus, $\nu=\pre_2(\lambda)=(x,x,1)$, so $\nu$ is a partition of $2x+1$. Thus, varying $x$ over $\ZZ^+$, it follows that any odd integer $n$ not equal to $1$ has the property that $\pre_2(n)\geq 2$. Now, take $\lambda=(x,1,1,1,1)$ for any positive integer $x$. Thus, $\nu=\pre_2(\lambda)=(x,x,x,x,1,1,1,1,1,1)$, so $\nu$ is a partition of $4x+6$. Thus, varying $x$ over non-negative integers (if $x=0$, let the partition $\lambda=(1,1,1,1,1,1)$) it follows that every $n$ congruent to $2\pmod{4}$ that is not equal to $2$ has the property that $\pre_2(n)\geq 2$.

Now suppose a given $n$ has the property that $\pre_2(n)\geq 2$ and let $\lambda=(\lambda_1,\lambda_2,\dots,\lambda_\ell)$ be the corresponding partition with $\pre_2(\lambda)=\nu$. If we consider the new partition $\lambda'=(2\lambda_1,2\lambda_2,\dots,2\lambda_\ell)$, it follows that $\pre_2(\lambda')=\nu'$ is a partition of $4n$, in which $\lambda'\neq (4n,1)$. Thus, if $n$ satisfies $\pre_2(n)\geq 2$, then $4n$ satisfies $\pre_2(4n)\geq 2$. 

Applying this statement inductively, it follows that $\pre_2(2^a\cdot X)\geq 2$ for all odd integers $X\geq 3$ and integers $a\geq 0$.

We now show that if $n=2^k$ for any $k\geq 3$, then $\pre_2(n)\geq 2$. Consider $\lambda=(2^{k-2}-1,2,2)$ for $k\geq 4$. Thus, $\nu=\pre_2(\lambda)=(2^{k-1}-2,2^{k-1}-2,4)$, so $\nu$ is a partition of $2^k$, in which $\lambda\neq 2^k$. If $k=3$, then $\lambda=(2,2,1)$ results in $\nu$ being a partition of $2^3=8$. Thus, we have shown all $n$ besides $1,2,4$ fail to satisfy $\pre_2(n)=1$.
\end{proof}

\section{Complete homogeneous polynomials}\label{sec: prh}
We first prove general injectivity for $\prh_k$. Note that Hadelyn et al.~\cite{Ha+} proved general injectivity as well, but we offer our own proof, independently derived, in order to highlight the relevance of this methodology.
\begin{theorem}
The function $\prh_k$ is injective on $\mathcal{P}(n)$.
\end{theorem}
\begin{proof}
We use a ``multiset" argument analogous to Lemma 2.1 in \cite{Li}.

In particular, take two partitions $\lambda,\mu$ such that $\prh_k(\lambda)=\prh_k(\mu)$. Note the following equality:
\[\ell(\prh_k(\lambda))=\binom{k+\ell(\lambda)-1}{k}=\binom{k+\ell(\mu)-1}{k}=\ell(\prh_k(\mu)).\]
Thus, as the binomial function is strictly increasing, it follows that $\ell(\lambda)=\ell(\mu)$.

Furthermore, note $\lambda_1^k=\mu_1^k$, as they are both maximum elements (there can exist multiple maximum elements) in $\prh_k(\lambda)$ and $\prh_k(\mu)$, respectively. Thus, $\lambda_1=\mu_1$.

We now proceed by induction on $\ell$, with our base case being $\lambda_1=\mu_1$. Suppose $\lambda_i=\mu_i$ for all $1\leq i\leq k'$. Consider the multiset $M_{k,k'}(\lambda):=\{\{\lambda_1^{p_1}\lambda_2^{p_2}\cdots\lambda_{k'}^{p_{k'}}: p_1+\dots+p_{k'}=k\}\}$. Thus, by the inductive hypothesis, $M_{k,k'}(\lambda)=M_{k,k'}(\mu)$. 

Thus, it follows:
\[\prh_k(\lambda)\setminus M_{k,k'}(\lambda)=\prh_k(\mu)\setminus M_{k,k'}(\mu).\]
Hence, maximum elements in both of these sets are equal. One can easily see a maximum element in both sets is $\lambda_1^{k-1}\lambda_{k'+1}$ and $\mu_1^{k-1}\mu_{k'+1}$, respectively. Therefore, $\lambda_{k'+1}=\mu_{k'+1}$, as desired.

Therefore, by the principle of strong induction, we conclude $\lambda_i=\mu_i$ for all $1\leq i\leq \ell$, so $\lambda=\mu$. Thus, $\prh_k$ is injective, as desired.
\end{proof}

\begin{remark}
One may note the relative ease of this proof compared to those of $\pre_k$. Particularly, as $h_k$ has monomials of one variable, it is much easier to extract information about $\lambda$ from $\prh_k(\lambda)$ compared to the problem of $\pre_k$.
\end{remark}

To highlight this, we demonstrate an algorithmic approach when $m_\lambda(1)>0$. 
\begin{lem}\label{lem: prhalg}
For any partition $\lambda$ such that $m_\lambda(1)>0$, we can extract $\lambda$ from $\nu=\prh_k(\lambda)$. 
\end{lem}
\begin{proof}
Using a ``stars and bars" argument, it follows:
\[m_\nu(1)=\binom{m_{\lambda}(1)+k-1}{k}.\]
Thus, as $m_\lambda(1)>0$, and due to the strictly increasing nature of the binomial function, it follows that we can determine $m_\lambda(1)$ from $m_\nu(1)$.

Now, consider a prime $p$. By logic similar to that used above, it follows
\[m_\nu(p)=m_\lambda(p)\cdot \binom{m_\lambda(1)+k-2}{k-1}.\]
Thus, we can extract $m_\lambda(p)$ from $m_\nu(p)$ for all prime $p$. 

We now proceed by induction. The base case is that we can retrieve $m_\lambda(1)$ from $m_\nu(1)$ and $m_\lambda(p)$ from $m_\nu(p)$ for all primes $p$, as demonstrated before. Now, fix a composite $y$, and assume that for all $x<y$, we can determine $m_\lambda(x)$ from $\nu$. Analogous to the argument made in Ballantine et al.~\cite{Bplus}, let $D_y$ be the collection of tuples $(d_1^{b_1},d_2^{b_2},\dots, d_t^{b_t})$ satisfying the following properties:
\begin{itemize}
\item $1<d_1<d_2<\dots<d_t<y$
\item $b_1,\dots, b_t>0$
\item $b_1+\dots+b_t\leq k$
\item $y=d_1^{b_1}d_2^{b_2}\dots d_t^{b_t}$
\end{itemize}
We obtain the following recursive relation:
\begin{align*}
m_\nu(y)&=m_\lambda(y)\cdot \binom{m_\lambda(1)+k-2}{k-1}\\
&+\sum_{(d_1^{b_1},\dots,d_t^{b_t})\in D_y} \binom{m_\lambda(1)+k-1-(b_1+b_2+\dots+b_t)}{m_\lambda(1)-1}\prod_{i=1}^t \binom{m_\lambda(d_i)+b_i-1}{b_i}.\\
\end{align*}
Hence, we can determine $m_\lambda(y)$ from $m_\nu(y)$. Thus, by the principle of strong induction, we can determine $m_\lambda(i)$ from $\nu$ for all positive integers $i$. Therefore, we can construct $\lambda$ from $\nu=\prh_k(\lambda)$, as desired.
\end{proof}

\begin{remark}\label{remark: prhalg}
If $m_\lambda(1)=0$, by Theorem~\ref{theorem: prh}, there does exist an algorithm as well. However, it would be of a different nature than the one above as $m_\lambda(1)>0$ is a necessary condition for it to work. 
\end{remark}

\section{Schur polynomials}\label{sec: prs}

\subsection{Main results}
A natural question that arises is why we choose to consider injectivity for $\mathcal{P}(n)$ rather than the general case of $\mathcal{P}$. To illustrate this necessity, we give counterexamples for which $\prs_{\lambda'}$ is not injective on $\mathcal{P}$.

\begin{theorem}\label{theorem: neccesityofPn}
Consider a partition $\lambda'$ whose Young diagram is a rectangle with at least $2$ rows. Then $\prs_{\lambda'}$ is not injective on $\mathcal{P}$. It follows that $\prs_{\lambda'}$ is not injective on $\mathcal{P}$ for infinitely many partitions $\lambda'$. 
\end{theorem}
\begin{proof}
Let $\lambda'=(k^n)$ where $k,n$ are positive integers and $n\geq 2$. Using the Jacobi-Trudi identities, it follows:
\[s_{\lambda'}(x_1,\dots,x_n)=(x_1\cdots x_n)^k.\]
Hence, consider the partitions $\lambda=(4,1,\dots,1)$ and $\mu=(2,2,1,\dots,1)$. Thus, $s_{\lambda'}(\lambda)=4^k=2^k\cdot 2^k=s_{\lambda'}(\mu)$, so $\prs_{\lambda'}$ is not injective, as desired. 
\end{proof}

Note that $e_k=s_{(1^k)}$ represents a rectangular partition with length at least $2$ if $k\geq 2$. Therefore, if we apply Theorem~\ref{theorem: neccesityofPn} to the case of $\pre_k$, we can see that it is necessary to work over $\Pn$ rather than just the general case of $\mathcal{P}$.

Generalizing the idea from Theorem~\ref{theorem: counter}, we can show the necessity of working over partitions of length strictly greater than $\ell(\lambda')$:

\begin{lem}\label{lem: counter}
Consider a partition $\lambda'$ whose Young diagram is a rectangle with at least $3$ rows. Then $\prs_{\lambda'}$ is not injective on $\mathcal{P}_\ell(n)$. It follows that $\prs_{\lambda'}$ is noninjective for infinitely many partitions where $\ell\geq 3$ is the length of $\lambda'$.
\end{lem}
\begin{proof}
Let $\lambda'=(k^m)$ where $k,m$ are positive integers with $m\geq 3$. Recall $s_{\lambda'}(x_1,\dots,x_m)=(x_1\cdots x_m)^k$. By Theorem \ref{theorem: counter}, for $m\geq 3$ and $n>m+15$, there exists $\lambda,\mu\in\mathcal{P}(n)$ of length $m$ such that $\pre_m(\lambda)=\pre_m(\mu)$.

Thus $\lambda_1\lambda_2\cdots\lambda_m=\mu_1\mu_2\dots\mu_m$. Therefore, it follows:
\[s_{\lambda'}(\lambda)=(\lambda_1\cdots \lambda_m)^k=(\mu_1\cdots \mu_m)^k=s_{\lambda'}(\mu).\] 
Hence, $\prs_\lambda'$ is not injective, as desired. 
\end{proof}

We now focus on proving Theorem~\ref{theorem: prs_(2,1)}, using an argument similar to \cite{Li}. We first give the following lemma, analogous to Lemma 2.1 in \cite{Li}:
\begin{lem}\label{lem: prs21li}
For $\lambda,\mu\in\mathcal{P}(n)$, if $\prs_{(2,1)}(\lambda)=\prs_{(2,1)}(\mu)$ and $\lambda_1=\mu_1$, then $\lambda=\mu$. 
\end{lem}
\begin{proof}
Recall from Example~\ref{exam: prs21} that
\[\ell(\prs_{(2,1)}(\lambda))=2\cdot\binom{\ell(\lambda)}{2}+2\cdot \binom{\ell(\lambda)}{3}=2\cdot\binom{\ell(\mu)}{2}+2\cdot \binom{\ell(\mu)}{3}=\ell(\prs_{(2,1)}(\mu)).\]
Thus, due to the strictly increasing nature of the binomial function, it follows that $\ell(\lambda)=\ell(\mu)$.

Using an inductive approach, it suffices to show if $\lambda_i=\mu_i$ for $1\leq i\leq k-1$, then $\lambda_{k}=\mu_{k}$ for $k\geq 2$. 

Recall, say by the Jacobi-Trudi identities, the Schur polynomial of $(2,1)$ over $\ell$ variables can be rewritten as the following:
\[s_{(2,1)}(x_1,x_2,\dots,x_\ell)=\sum_{1\leq i< j\leq \ell} (x_i^2x_j+x_ix_j^2)+2\sum_{1\leq i<j<k\leq \ell}{x_ix_jx_k}.\]
By the inductive hypothesis, the summands of $s_{(2,1)}(\lambda_1,\lambda_2,\dots,\lambda_{k-1})$ and $s_{(2,1)}(\mu_1,\mu_2,\dots,\mu_{k-1})$ are the same. 

Thus, consider the summands of the polynomial
\[s_{(2,1)}(x_1,x_2,\dots,x_\ell)-s_{(2,1)}(x_1,x_2,\dots,x_{k-1})\]
when represented as a partition for $k\geq 3$. When evaluated at $\lambda$ and $\mu$, we can see that maximal elements would be $\lambda_1^2\lambda_{k}$ and $\mu_1^2\mu_{k}$. Thus, as $\prs_{(2,1)}(\lambda)=\prs_{(2,1)}(\mu)$ and by the inductive hypothesis, this means $\lambda_{k}=\mu_{k}$ if $k\geq 3$, as desired.

Thus, it suffices to show $\lambda_1=\mu_1$ implies $\lambda_2=\mu_2$. Note that maximum elements of $\prs_{(2,1)}(\lambda)$ and $\prs_{(2,1)}(\mu)$ are $\lambda_1^2\lambda_2$ and $\mu_1^2\mu_2$, respectively. Thus, as these two are equal, it follows $\lambda_2=\mu_2$, as desired.
\end{proof}

To finish the proof, recall the following fact \cite{Li}:
\[\lambda_1=\lim_{p\to\infty} \left(\sum_{i=1}^\ell \lambda_i^p\right)^{\frac1p}.\]
We claim that $\sum_{i=1}^\ell \lambda_i^{3^a}=\sum_{i=1}^\ell \mu_i^{3^a}$, which would, by the above equality, imply $\lambda_1=\mu_1$. We do so through induction on $a$. Our base case is $a=0$, where we have $\sum_{i=1}^\ell \lambda_i=\sum_{i=1}^\ell \mu_i=n$, which is true. Now, for the inductive step, assume $\sum_{i=1}^\ell \lambda_i^{3^k}=\sum_{i=1}^\ell \mu_i^{3^k}$. Therefore, by the formulation of $s_{(2,1)}$, it follows:
\begin{align*}
\sum_{1\leq i\leq \ell}\lambda_i^{3^{k+1}}&=\left(\sum_{1\leq i\leq \ell}\lambda_i^{3^{k}}\right)^3-3\sum_{1\leq i\neq j\leq \ell} (\lambda_i^2\lambda_j)^{3^{k}}-6\sum_{1\leq i<j<k\leq \ell}(\lambda_i\lambda_j\lambda_k)^{3^k}\\
&=\left(\sum_{1\leq i\leq \ell}\mu_i^{3^{k}}\right)^3-3\sum_{1\leq i\neq j\leq \ell} (\mu_i^2\mu_j)^{3^{k}}-6\sum_{1\leq i<j<k\leq \ell}(\mu_i\mu_j\mu_k)^{3^k}=\sum_{1\leq i\leq \ell}\mu_i^{3^{k+1}}.\\
\end{align*}
Therefore, by the principle of induction, we conclude that $\lambda_1=\mu_1$, so $\prs_{(2,1)}$ is injective on $\Pn$.

\prsTheorem*
\begin{remark}\label{remark: s21good}
The reason this argument works so well is because of how $s_{(2,1)}$ fundamentally acts. Particularly, $s_{(2,1)}$ can be written solely in terms of power sums that are powers of $3$:
\[s_{(2,1)}=\frac13(p_1^3-p_3).\]
Thus, this inductive-type method works perfectly, similar to the one used in \cite{Li}. 

However, the method does not readily extend to general $\lambda'$. 
\end{remark}

We now provide a lemma similar in nature to Theorem~\ref{theorem: length}.
\begin{lem}\label{lem: prslength}
For any partition $\lambda$ and integer $n\geq \lambda_1$, we have $(x_1x_2\dots x_m)^n\cdot s_\lambda(x_1^{-1},\dots,x_m^{-1})=s_{f_n(\lambda)}(x_1,\dots,x_m)$ where $f_n(\lambda_i)=n-\lambda_{m+1-i}$.
\end{lem}
\begin{proof}
We can use the Jacobi-Trudi identities: recall $s_\lambda=\frac{\det(x_j^{\lambda_i+m-i})_{i,j=1}^m}{\det(x_j^{m-i})_{i,j=1}^m}$. We look at the numerator and denominator of this expression separately. Note, in the denominator, if we were to replace $x_i$ with its reciprocal, it would become $\det(x_j^{i-m})_{i,j=1}^m$. Thus, it follows:
\[(x_1\dots x_m)^{m-1}\cdot \det(x_j^{i-m})_{i,j=1}^m=\det(x_j^{i-1})_{i,j=1}^m=(-1)^{\binom{m}{2}}\cdot \det(x_j^{m-i})_{i,j=1}^m.\]
Likewise, if we replaced $x_i$ with its reciprocal in the numerator and multiplied it by $(x_1\dots x_m)^{n+m-1}$, we would get the following equality:
\begin{align*}
(x_1\dots x_m)^{n+m-1}\det(x_j^{-\lambda_i-m+i})_{i,j=1}^m&=\det(x_j^{n-\lambda_i+i-1})_{i,j=1}^m\\
&=\det(x_j^{f_n(\lambda_{m+1-i})+i-1})_{i,j=1}^m\\
&=(-1)^{\binom{m}{2}}\det(x_j^{f_n(\lambda_i)+m-i})_{i,j=1}^m.
\end{align*}

Therefore, it follows:
\begin{align*}
(x_1\dots x_m)^n\cdot s_\lambda(x_1^{-1},\dots,x_m^{-1})&=\frac{(x_1\dots x_m)^{n+m-1}\det(x_j^{-\lambda_i-m+i})_{i,j=1}^m}{(x_1\dots x_m)^{m-1}\det(x_j^{-m+i})_{i,j=1}^m}\\
&=\frac{(-1)^{\binom{m}{2}}\det(x_j^{f_n(\lambda_i)+m-i})_{i,j=1}^m}{(-1)^{\binom{m}{2}}\det(x_j^{m-i})_{i,j=1}^m}\\
&=s_{f_n(\lambda)}(x_1,\dots,x_m),
\end{align*}
as desired.
\end{proof}

Note, by the symmetry of $s_{\lambda'}$, if $\prs_{\lambda'}(\lambda)=\prs_{\lambda'}(\mu)$, then $\prod \lambda_i=\prod\mu_i$. Therefore, it follows that $s_{f_n(\lambda')}(\lambda)=s_{f_n(\lambda')}(\mu)$, potentially allowing us to reduce certain partitions to easier to access partitions. Hence, we can make the following statement:
\begin{cor}
If $\prs_{(m,m-1,\dots,1)}(\lambda)=\prs_{(m,m-1,\dots,1)}(\mu)$ where $\lambda,\mu$ are each of length $m$, then $\prs_{(m-1,\dots,1)}(\lambda)=\prs_{(m-1,\dots,1)}(\mu)$. 
\end{cor}

So far, we have been thinking about Schur polynomials from the perspective of the Jacobi-Trudi identities. However, it is helpful to think about it visually, from the perspective of their semi-standard Young tableaux. To demonstrate this, we make the following theorem, a generalization of Lemma~\ref{lem: prs21li}:

\begin{theorem}\label{theorem: liextensionprs}
For $\lambda,\mu\in\mathcal{P}(n)$ of the same length, if $\prs_{\lambda'}(\lambda)=\prs_{\lambda'}(\mu)$ and $\lambda_i=\mu_i$ for all $1\leq i\leq k-1$, then $\lambda=\mu$, where $k$ is the length of $\lambda'$.
\end{theorem}
\begin{proof}
Consider any integer $k'\geq k$, which we let be our inductive index. Using an inductive approach, it suffices to show if $\lambda_i=\mu_i$ for $1\leq i\leq k'-1$, then $\lambda_{k'}=\mu_{k'}$. 

First, note that a maximum element of $\prs_{\lambda'}({\lambda})$ would be generated by the SSYT of $\lambda'$ such that the first row consists of only $1$'s, the second row consists of only $2$'s, and so on with the $k$th row consisting of $\lambda'_k$ $k$'s. Thus, a maximum element of $\prs_{\lambda'}{\lambda}=\lambda_1^{\lambda'_1}\lambda_2^{\lambda'_2}\cdots\lambda_{k-1}^{\lambda'_{k-1}}\lambda_k^{\lambda'_k}$. Likewise, a maximum element of 
$\prs_{\lambda'}{\mu}=\mu_1^{\lambda'_1}\mu_2^{\lambda'_2}\cdots\mu_{k-1}^{\lambda'_{k-1}}\mu_k^{\lambda'_k}$. Thus, as both of these expressions are equal, and by the hypothesis that for all $1\leq i\leq k-1$, we have $\lambda_i=\mu_i$, we can determine that $\lambda_k=\mu_k$.

Now, by the inductive hypothesis, it follows that the summands of $s_{\lambda'}(\lambda_1,\lambda_2,\dots,\lambda_{k'-1},0)$ and $s_{\lambda'}(\mu_1,\mu_2,\dots,\mu_{k'-1},0)$ are the same. 
Thus, consider the summands of the polynomial
\[s_{\lambda'}(x_1,x_2,\dots,x_\ell)-s_{\lambda'}(x_1,x_2,\dots,x_{k'-1},0)\]
when represented as a partition. Consider the SSYT of $\lambda'$ such that the first row consists of only $1$'s, the second row consists of only $2$'s, and so on, except with the $k$th row consisting of $\lambda'_k-1$ $k$'s and then one $k'$. It follows that this particular SSYT represents a maximum element of $s_{\lambda'}(x_1,x_2,\dots,x_\ell)-s_{\lambda'}(x_1,x_2,\dots,x_{k'-1},0)$. In particular, when evaluated at $\lambda$ and $\mu$, these maximum elements would be equal:
\[\lambda_1^{\lambda'_1}\lambda_2^{\lambda'_2}\cdots\lambda_{k-1}^{\lambda'_{k-1}}\lambda_k^{\lambda'_k-1}\lambda_{k'}=\mu_1^{\lambda'_1}\mu_2^{\lambda'_2}\cdots\mu_{k-1}^{\lambda'_{k-1}}\mu_k^{\lambda'_k-1}\mu_{k'}.\]
Thus, by the inductive hypothesis, it follows that $\lambda_k^{\lambda'_k-1}\lambda_{k'}=\mu_k^{\lambda'_k-1}\mu_{k'}$. As $\lambda_k=\mu_k$, it then follows that $\lambda_{k'}=\mu_{k'}$, as desired.
\end{proof}

As shown by the above proof, the geometry of a standard partition can be quite restrictive, particularly by the second condition in the definition of a SSYT. Therefore, we are motivated to look beyond standard partitions into skew shapes.

\subsection{Skew shapes}

\skewTheorem*
\begin{proof}
Suppose $\lambda'/\mu'$ is of size $m$. Thus, over $k\geq 1$ variables, it follows:
\[s_{\lambda'/\mu'}(x_1,x_2,\dots,x_k)=(x_1+x_2+\dots+x_k)^m.\]
Due to the similarity in structure this has with the complete homogeneous polynomials, we can readily apply the same argument used in Theorem~\ref{theorem: prh}, allowing us to prove injectivity of $\prs_{\lambda'/\mu'}$. 
\end{proof}
\begin{cor}\label{cor: skewcor}
The function $\prs_{(m,m-1,\dots,1)/(m-1,m-2,\dots,1)}$ is injective on $\Pn$.
\end{cor}
Therefore, it seems that the two definitions of $s_\lambda$, through the Jacobi-Trudi identities and the semi-standard Young tableaux formulation, seem to both be very useful in the development of $\prs_{\lambda'}$.

\subsection{Applications}\label{subsec:app}
Given symmetric functions $f,g$ where $g = \sum_a x^{a^i}$ where each $a^i$ is a degree vector and repeated monomials appear multiple times, then the plethysm $f[g]$ of symmetric functions is defined to be $f(x^{a_1}, x^{a_2}, ...)$. Li's injectivity result for $\pre_2$~\cite{Li} shows that the family of plethysms $\{e_r[e_2](\lambda)\}_{r\ge 1}$, which determines the multiset $\pre_2(\lambda)$, also in fact determines $\lambda$ and hence $\{e_r(\lambda)\}_{r \ge 1}$ as observed in \cite{Bplus}.  However, due to the failure of injectivity for $\pre_k$ for $k \ge 3$, the family of plethysms $\{e_r[e_k]\}$ evaluated at the parts of a partition $\lambda$ does not necessarily determine the entire family $\{e_r\}$ evaluated at $\lambda$. 

Furthermore, based on the result that $\prs_{(2,1)}$ is injective on $\mathcal{P}(n)$, as shown in Theorem~\ref{theorem: prs_(2,1)}, it follows that the family of plethysms $\{s_\nu[s_{(2,1)}]\}_{\nu\in\mathcal{P}}$ evaluated at $\lambda$ a partition of $n$ determines $\lambda$ and hence the original Schur polynomials evaluated at $\lambda$.

The problem of checking whether $\prs_\mu(\lambda)$ is injective on the set $\mathcal{P}(n)$ is equivalent to building a diagonal matrix $M_\lambda = \mathrm{diag}(\lambda_1,\dots,\lambda_\ell)$, and asking whether $M_\lambda$ can be retrieved (up to ordering of eigenvalues) from the spectrum, regarded as a multiset, of $\mathbb{S}_\mu(M_\lambda)$, where $\mathbb{S}_\mu(-)$ is the Schur functor, with the additional constraint that the diagonal entries of $M_\lambda$ are positive integers with sum $n$.

\section{Open Questions}\label{sec: openquestions}

Even though the $\pre_k$ injectivity conjecture of \cite{BBM} was shown by several authors to be false, a variant of it remains open.
\mainQuest*

Our work and \cite{Ha+} lead us to the following refinement.
\begin{conj}\label{conj: new}
For each integer $k\geq 3$, there exists some constant $M_k$ such that $\pre_k$ fails to be injective on $\mathcal{P}_{>k}(n)$ for all $n\geq M_k$.
\end{conj}

The result of \cite{Garg}, proven independently by us, that $\pre_2(n) = 1$ if and only if $n \in \{1,2,4\}$ motivates the following question.
\begin{quest}\label{quest: deNew}

For what positive integers $a$ do there exist $n$ and $k$ such that $\pre_k(n)=a$?

\end{quest}

Another fruitful direction concerns $\prs_{\lambda'}$. We refine Question 5.3 of \cite{Ha+} as follows. 
\begin{conj}\label{conj: prsNew}
For every two-row partition $\lambda'$, the function $\prs_{\lambda'}$ is injective on $\mathcal{P}(n)$ for every positive integer $n$ on which $\prs_{\lambda'}$ is defined.
\end{conj}

The case $\lambda' = (1,1)$ was shown in \cite{Li}, but the general proof will probably not use exactly the same techniques. 
However, proving this conjecture would have a nice plethysm implication analogous to that of $\prs_{(2,1)}$ in Section \ref{subsec:app}, except with a two-row partition instead of $(2,1)$. It would also be interesting to consider the case of $\lambda'$ a staircase or hook partition. 

For reference, the code for this project can be found in the GitHub repository linked \href{https://github.com/katherineatung/primes-2026}{here}.

\section{Acknowledgments}\label{sec: acknowledgments}
We thank the MIT PRIMES program, during which this research was conducted. K.T. thanks the Max Planck Institute for Mathematics in the Sciences for their hospitality and access to the compute server. R.T. thanks his family, friends, and teachers for their support. 
\printbibliography
\end{document}